\documentclass{amsart}
\usepackage{amsmath,amsthm, latexsym, amssymb,mathrsfs,  tikz-cd}
\usepackage{stackengine,scalerel}
\usepackage[colorlinks]{hyperref}
\usepackage[capitalize]{cleveref}
\definecolor{dark-red}{rgb}{0.6,0,0}
\definecolor{dark-green}{rgb}{0,0.4,0}
\definecolor{medium-blue}{rgb}{0,0,0.5}
\hypersetup{
    colorlinks, linkcolor={dark-red},
    citecolor={dark-green}, urlcolor={medium-blue}
}

\newcommand{\Fil}{\mr{Fil}}

\newcommand{\Id}{\mr{Id}}

\newcommand{\mbb}[1]{\mathbb{#1}}

\newcommand{\mr}[1]{\mathrm{#1}}

\newcommand{\Sym}{\mathrm{Sym}}

\numberwithin{equation}{subsection}
\numberwithin{equation}{subsubsection}

\newcommand{\Rep}{\mathrm{Rep}}
\newcommand{\Sp}{\mathrm{Sp}}

\newcommand{\mult}{\mr{mult}}

\theoremstyle{plain}

\newtheorem*{theorem*}{Theorem}

\newtheorem{theorem}[subsubsection]{Theorem}
\newtheorem{corollary}[subsubsection]{Corollary}

\newtheorem{lemma}[subsubsection]{Lemma}

\theoremstyle{definition}

\newtheorem{example}[subsubsection]{Example}

\newcommand{\Exp}{\mathrm{Exp}}

\title{The negative $\sigma$-moment generating function}

\author{Sean Howe}

\begin{document}

\begin{abstract}
For $X$ a pre-$\lambda$ random variable, we show the $\sigma$-moment generating function of $-X$ can be obtained from the $\sigma$-moment generating function of $X$ by applying the composition of the standard and degree flip involutions on symmetric power series. This isometric involution is natural as it preserves the pre-$\lambda$ ring structure on symmetric power series with pre-$\lambda$ coefficients, thus this formula provides a simple description of the $\sigma$-moment generating function of $-X$ whenever the $\sigma$-moment generating function of $X$ has a simple description using the pre-$\lambda$ structure. As an application we compute, in a natural range, the dimensions of orthogonal and symplectic group invariants in tensor products of exterior powers of their standard representations on $\mathbb{C}^n$. We also compute a generating function for stable traces of Frobenius related to the moment conjecture for prime-order function field Dirichlet characters.
\end{abstract}

\maketitle
\tableofcontents

\section{Introduction}

In \cite{Howe.RandomMatrixStatisticsAndZeroesOfLFunctionsViaProbabilityInLambdaRings}, we introduced a theory of probability in pre-$\lambda$ rings in order to give concise tools for describing the distributions of random variables valued in multisets of complex numbers and other similar structures. For $(R, \mathbb{E}: R \rightarrow C)$ a pre-$\lambda$ probability space as in \cite[Definition 3.1.1]{Howe.RandomMatrixStatisticsAndZeroesOfLFunctionsViaProbabilityInLambdaRings}, the $\Lambda$-distribution of a random variable $X \in R$ is encoded by the $\sigma$-moment generating function of \cite[Definition 3.2.1]{Howe.RandomMatrixStatisticsAndZeroesOfLFunctionsViaProbabilityInLambdaRings}, 
\[ \mathbb{E}[\Exp_{\sigma}(Xh_1)] \in \Lambda_C^\wedge, \]
where $\Exp_{\sigma}$ is the $\sigma$-exponential and $\Lambda_C^\wedge$ is a ring of symmetric power series with coefficients in $C$. In practice, this $\sigma$-moment generating function often provides a simple description of the $\Lambda$-distribution of $X$ (cf. \cite[Theorems A-C, \S3.3]{Howe.RandomMatrixStatisticsAndZeroesOfLFunctionsViaProbabilityInLambdaRings}).  

One lacuna of the theory developed in \cite{Howe.RandomMatrixStatisticsAndZeroesOfLFunctionsViaProbabilityInLambdaRings}, as noted already in \cite[Remark 8.1.1]{Howe.RandomMatrixStatisticsAndZeroesOfLFunctionsViaProbabilityInLambdaRings}, is that it is not obvious how to compute the $\sigma$-moment generating function for the negative $-X$ of a random variable $X\in R$ in a concise way starting from the $\sigma$-moment generating function of $X$. In this work, we provide such a concise formula: for $\tilde{\omega}$ the composition of the standard involution $\omega$ on $\Lambda$ and the degree flip involution (defined carefully in \cref{ss.the-involution}), we show in \cref{theorem.expectation-of-negative} that 
\begin{equation}\label{eq.intro-neg-exp}\mathbb{E}[\Exp_{\sigma}(-X h_1)] = \tilde{\omega}\left( \mathbb{E}[\Exp_{\sigma}(X h_1)] \right). \end{equation}

The involution $\tilde{\omega}$ is more natural in our context than the standard involution $\omega$ as, for $R$ a torsion-free pre-$\lambda$ ring, $\tilde{\omega}$ respects the pre-$\lambda$ ring structure on $\Lambda_R^\wedge$ (\cref{lemma.involution-map-of-pre-lambda-rings})\footnote{When $R=\mathbb{Z}$, this involution is introduced in  \cite[\S7.20]{GetzlerKapranov.ModularOperads}, and a version of \cref{lemma.involution-map-of-pre-lambda-rings} is implicit in the proof of  \cite[Corollary 8.19]{GetzlerKapranov.ModularOperads}. We thank Dan Petersen for pointing out this reference!}. In particular, the right-hand side of \cref{eq.intro-neg-exp} can be computed in a concise way whenever the $\sigma$-moment generating function of $X$ itself has a concise formula using pre-$\lambda$ operations such as the $\sigma$-exponential or motivic powers. 

\begin{example}
Suppose $C$ (the codomain of $\mathbb{E}$) is a pre-$\lambda$ ring and $X \in R$ is a pre-$\lambda$ binomial random variable with parameters $p$ and $N$ in $C$ as in \cite[\S 3.3]{Howe.RandomMatrixStatisticsAndZeroesOfLFunctionsViaProbabilityInLambdaRings}, i.e. 
\[ \mathbb{E}[\Exp_{\sigma}(X)]=(1+p (h_1 + h_2 + h_3+ \ldots))^N, \]
where the power is taken using the associated power structure as in \cite[\S 2.6]{Howe.RandomMatrixStatisticsAndZeroesOfLFunctionsViaProbabilityInLambdaRings}. Then, the results described above imply that 
\begin{align*} \mathbb{E}[\Exp_{\sigma}(-X)]&=\tilde{\omega}((1+p (h_1 + h_2 + h_3+ \ldots))^N)\\&=(\tilde{\omega}(1+p (h_1 + h_2 + h_3+ \ldots)))^N\\&=(1 + p(-e_1+e_2-e_3+\ldots))^N. 
\end{align*}
As noted in \cite[Remark 8.1.1]{Howe.RandomMatrixStatisticsAndZeroesOfLFunctionsViaProbabilityInLambdaRings}, this can be used to reduce the number of computations used to establish the asymptotic distributions of $L$-functions of hypersurface sections in \cite[Section 8]{Howe.RandomMatrixStatisticsAndZeroesOfLFunctionsViaProbabilityInLambdaRings} --- see in particular \cite[Theorem 8.3.1]{Howe.RandomMatrixStatisticsAndZeroesOfLFunctionsViaProbabilityInLambdaRings}, where, lacking \cref{theorem.expectation-of-negative}, we had to compute the moment generating functions of a random variable and its negative by running the same method twice over. 
\end{example}

It is particularly interesting to apply this formula to the random matrix $\Lambda$-distributions for orthogonal and symplectic groups computed in \cite[Theorem A]{Howe.RandomMatrixStatisticsAndZeroesOfLFunctionsViaProbabilityInLambdaRings}. As an application, we obtain in \cref{s.invariants-random-matrices} a computation of the dimensions of orthogonal and symplectic invariants in tensor products of exterior powers of $\mathbb{C}^n$ that is novel in that its only representation-theoretic input is the classical invariant theory for polynomial functions (see \cref{remark.explicit-formulas-invariants} for further discussion).

\begin{example}
When $X$ is a random variable whose values are the multisets of reciprocal roots $\alpha$ of arithmetic $L$-functions  as in \cite[Theorems B and C]{Howe.RandomMatrixStatisticsAndZeroesOfLFunctionsViaProbabilityInLambdaRings}, $\mathbb{E}[\Exp_{\sigma}(-X)]$ is the generating function for the infinite formal shifted moments, 
\[ \mathbb{E}[\Exp_{\sigma}(-X)]=\mathbb{E}\left[\prod_{i=1}^\infty \prod_\alpha (1-\alpha t_i) \right]. \]
This generating function is natural in the context of the function field moment conjecture for quadratic characters established in \cite{BergstromDiaconuPetersenWesterland.HyperellipticCurvesTheScanningMapAndMomentsOfFamiliesOfQuadraticLFunctions}.  \cref{theorem.expectation-of-negative} thus explains more conceptually the translation, described in \cite[Example 7.2.3]{Howe.RandomMatrixStatisticsAndZeroesOfLFunctionsViaProbabilityInLambdaRings}\footnote{In \cite[Example 7.2.3]{Howe.RandomMatrixStatisticsAndZeroesOfLFunctionsViaProbabilityInLambdaRings} we used the involution $\omega$, but this is equivalent to the use of $\tilde{\omega}$ in this special case because all of the symmetric functions appearing in the computation are even.}, between the $\sigma$-moment generating function computed in the quadratic case of \cite[Theorem B]{Howe.RandomMatrixStatisticsAndZeroesOfLFunctionsViaProbabilityInLambdaRings} and the generating function of \cite[Theorem 11.2.3]{BergstromDiaconuPetersenWesterland.HyperellipticCurvesTheScanningMapAndMomentsOfFamiliesOfQuadraticLFunctions}. \cref{theorem.expectation-of-negative} can be used to translate all of the $\sigma$-moment generating functions computed in \cite[Theorems B and C]{Howe.RandomMatrixStatisticsAndZeroesOfLFunctionsViaProbabilityInLambdaRings} into generating functions for  infinite formal shifted moments in those cases.  

In particular, in the case of $\ell$-power characters in \cite[Theorem B]{Howe.RandomMatrixStatisticsAndZeroesOfLFunctionsViaProbabilityInLambdaRings}, one obtains the following generating series for stable traces of Frobenius. For $\ell | (q-1)$, let $U_d/\mathbb{F}_q$ denote the space of $\ell$-power free degree $d$ polynomials and, for $\chi$ a fixed non-trivial character of $\mu_\ell(\mathbb{F}_q)$, let $V_d$ be the rank one Kummer local system on the non-vanishing locus $U_d'\subseteq U_d$ obtained by pushing out the $\ell$th root cover by $\chi$. For $\tau$ a partition, we write $S^\tau$ for the associated Schur functor. 
We write 
\[ [H_\bullet^{\mathrm{stable}}( S^{\tau_1}V \otimes S^{\tau_2}V^*)] : = \lim_{d \rightarrow \infty} \sum_i (-1)^i [H^i_c(U_d', S^{\tau_1} V_d \otimes S^{\tau_2}V_d^*)(d)]\]
where the brackets on the Tate-twisted compactly supported cohomology groups denote the characteristic power series of geometric Frobenius in $W(\mathbb{C})=1+t\mathbb{C}[[t]]$ 
and the convergence is treated coefficient-wise. Then, from \cite[Theorem B]{Howe.RandomMatrixStatisticsAndZeroesOfLFunctionsViaProbabilityInLambdaRings} and \cref{theorem.expectation-of-negative}, we deduce the following generalization of \cite[Example 7.2.3]{Howe.RandomMatrixStatisticsAndZeroesOfLFunctionsViaProbabilityInLambdaRings}:
\begin{multline*} \sum_{\tau_1,\tau_2} \left[H_\bullet^{\mathrm{stable}}( S^{\tau_1}V \otimes S^{\tau_2}V^*)\right](-1)^{|\tau_1|+|\tau_2|}s_{\tau_1'}\overline{s}_{\tau_2'}=\\ \mathrm{Exp}_{\sigma}\left(\mathrm{Log}_{\sigma}\left([q^{-{l+1}}]+ \ldots +[q^{-1}] +\sum_{\substack{k_1,k_2 \geq 0 \\ k_1 \equiv k_2 \mod \ell}}h_{k_1}\overline{h}_{k_2} \right)[q]-1\right). \end{multline*}
Here $\tau'$ denotes the conjugate partition to $\tau$ and $s_{\tau'}$ denotes the associated Schur symmetric function. This series may be useful in stable cohomology computations. 

\end{example}

\subsection{Acknowledgements} During the preparation of this work, Sean Howe was supported by the National Science Foundation through grant DMS-2201112. We were encouraged to revisit the computation of the negative $\sigma$-moment generating function by discussions at an AIM workshop on ``Moments in families of L-functions over function fields" (see, e.g., \cref{example.moments}) --- we thank all of the organizers and participants, and especially Dan Petersen for a particularly helpful conversation.

\section{Main results}

We will use the notation of \cite[\S2]{Howe.RandomMatrixStatisticsAndZeroesOfLFunctionsViaProbabilityInLambdaRings} for symmetric functions, symmetric power series, and pre-$\lambda$ rings. In particular, we recall that, when $R$ is a pre-lambda ring, the symmetric power series ring $\Lambda_R^\wedge$ also has a natural pre-$\lambda$ ring structure. 

\subsection{The involution}\label{ss.the-involution}
For $R$ a ring, we let $\omega$ denote the standard involution on $\Lambda_R$ swapping the $i$th elementary symmetric function $e_i$ with the $i$th complete symmetric function $h_i$. We let $\tilde{\omega}$ denote the composition of $\omega$ with the degree flip involution which acts as $(-1)^i$ on the degree $i$ component. Comparing with the usual formulas for $\omega$ (see, e.g., \cite[\S2.8]{Howe.RandomMatrixStatisticsAndZeroesOfLFunctionsViaProbabilityInLambdaRings}), we have, for $i \geq 1$,
\[ \tilde{\omega}(e_i)=(-1)^i e_i,\; \tilde{\omega}(h_i)=-1^i h_i,\; \textrm{ and } \tilde{\omega}(p_i)=-p_i. \]
These formulas shows $\tilde{\omega}$ is also an involution. It extends naturally to $\Lambda_R^\wedge$. 

\begin{lemma}\label{lemma.involution-map-of-pre-lambda-rings}
    If $R$ is a torsion-free pre-$\lambda$ ring, then $\tilde{\omega}: \Lambda_R^\wedge \rightarrow \Lambda_R^\wedge$ is a map of pre-$\lambda$ rings. In particular, for $f$ in $\Fil^1 \Lambda_R^\wedge$,
    \[ \tilde{\omega} \left( \Exp_{\sigma}(f) \right)= \Exp_{\sigma}(\tilde{\omega}(f)).\]
\end{lemma}
\begin{proof}
Replacing $R$ with $R \otimes \mathbb{Q}$, we may assume $R$ is a $\mathbb{Q}$-algebra.  Then, to verify $\tilde{\omega}$ respects the pre-$\lambda$ structure, it suffices to check the Adams operations $p_i \circ$ commute with $\tilde{\omega}$.  Since $R$ is a $\mathbb{Q}$-algebra, any element of $\Lambda_R^\wedge$ is of the form $\sum_\tau r_\tau p_{\tau}$ for $r_\tau \in R$ and $p_\tau=p_{\tau_1} p_{\tau_2} \ldots$ the usual power sum monomials. We then have
\[ p_i \circ \sum_\tau r_\tau p_{\tau} = \sum_\tau (p_i \circ r_\tau) p_{i\tau}. \]
where $i\tau$ denotes the multiplication of each term in the partition $\tau$ by $i$ (e.g. $i (2^3 5^1)=(2i)^3(5i)^1$). The constant term associated to the empty partition $\tau=\emptyset$ on each side of the equation is $r_\emptyset$; thus we may subtract $r_\emptyset$ off at the beginning of the computation to assume below that $r_\emptyset=0$. Then we have 
\begin{align*} \tilde{\omega} ( p_i \circ \sum_\tau r_\tau p_{\tau} ) = \sum_\tau (p_i \circ r_\tau) (-1)^{||i\tau||} p_{i\tau} \textrm{, and } \\  p_i \circ \tilde{\omega}(\sum_\tau r_\tau p_{\tau} )= \sum_{\tau} (p_i \circ r_\tau) (-1)^{||\tau||} p_{i\tau} \end{align*}  
for $||\tau||$ counting the number of terms in a partition (e.g. $|2^3 5^1|=3+1=4$). We conclude as $||i\tau||=||\tau||$. 
\end{proof}

Both of the involutions $\tilde{\omega}$ and $\omega$ extend to $\Lambda_R^\wedge$. We write $\langle, \rangle$ for the extension of the Hall inner product on $\Lambda_R$ to a pairing between $\Lambda_R^\wedge$ and $\Lambda_R$ (cf. \cite[2.7]{Howe.RandomMatrixStatisticsAndZeroesOfLFunctionsViaProbabilityInLambdaRings}). 

\begin{lemma}\label{lemma.involution-preserves-inner-product}
    The involution $\tilde{\omega}$ is an isometry, i.e., for $f \in \Lambda_R^\wedge$ and $g \in \Lambda_R$, 
    \[ \langle \tilde{\omega}(f), \tilde{\omega}(g) \rangle = \langle f , g \rangle. \]
\end{lemma}
\begin{proof}
    We write $f=\sum f_i$ and $g = \sum g_i$ as sums of homogeneous components, where $f_i$ and $g_i$ are of degree $i$. Then,
    \begin{multline} \langle \tilde{\omega}(f), \tilde{\omega}(g) \rangle = \langle \sum_i (-1)^i \omega(f_i), \sum_i (-1)^i \omega(g_i) \rangle =\\ \sum_i \langle \omega(f_i), \omega(g_i) \rangle = \sum_i \langle f_i, g_i \rangle = \langle \sum_i f_i, \sum_i g_i \rangle=\langle f, g \rangle, \end{multline}
    where here we have used both that $\omega$ preserves the Hall inner product and that terms in different degrees are orthogonal. 
\end{proof}

\subsection{The negative $\sigma$-moment generating function} 

The following is the main result. We give two proofs --- the first uses the Hall inner product to extract the desired values of the $\Lambda$-distribution from the $\sigma$-moment generating function as in \cite[Lemma 3.2.2]{Howe.RandomMatrixStatisticsAndZeroesOfLFunctionsViaProbabilityInLambdaRings} and then combines this with \cref{lemma.involution-preserves-inner-product}. This is the ``obvious" strategy from the perspective of \cite{Howe.RandomMatrixStatisticsAndZeroesOfLFunctionsViaProbabilityInLambdaRings}. The second, shorter proof uses \cref{lemma.involution-map-of-pre-lambda-rings}. 

\begin{theorem}\label{theorem.expectation-of-negative}
    Suppose $(\mathbb{E}, R)$ is a pre-$\lambda$ probability space. For $X \in R$, 
    \[ \mathbb{E}[\Exp_\sigma(-X h_1)]=\tilde{\omega}(\mathbb{E}[\Exp_\sigma (X h_1)]). \]
\end{theorem}
\begin{proof}[First proof of \cref{theorem.expectation-of-negative}]
    We first note that 
    \[ \mathbb{E}[\Exp_\sigma(-X h_1)] = \sum_{\tau} \mathbb{E}[ h_\tau \circ (-X)] m_\tau= \sum_{\tau} \mathbb{E}[e_\tau \circ X](-1)^{|\tau|} m_\tau. \]
    
    Now, by \cite[Lemma 3.2.2]{Howe.RandomMatrixStatisticsAndZeroesOfLFunctionsViaProbabilityInLambdaRings},
    \[ \mathbb{E}[e_\tau \circ X] = \langle \mathbb{E}[\Exp_\sigma (X h_1)], e_\tau \rangle. \]
    Applying the identity of \cref{lemma.involution-preserves-inner-product} to the right-hand side of this equation yields
    \[ \mathbb{E}[e_\tau \circ X] = \langle \tilde{\omega}(\mathbb{E}[\Exp_\sigma (X h_1)]), (-1)^{|\tau|} h_\tau \rangle \]
    and thus 
    \[  \mathbb{E}[e_\tau \circ X] (-1)^{|\tau|} = \langle \tilde{\omega}(\mathbb{E}[\Exp_\sigma (X h_1)]),  h_\tau \rangle. \]
    Thus, because $\langle m_\tau, h_\sigma \rangle = \delta_{\tau\sigma}$, 
    \[  \tilde{\omega}(\mathbb{E}[\Exp_\sigma (X h_1)]) = \sum_{\tau} \mathbb{E}[e_\tau \circ X] (-1)^{|\tau|}  m_\tau = \mathbb{E}[\Exp_\sigma(-X h_1)].\]
\end{proof}
\begin{proof}[Second proof of \cref{theorem.expectation-of-negative}]
    We have 
    \begin{align*} \mathbb{E}[\Exp_{\sigma}(-X h_1)] &= \mathbb{E}[\Exp_{\sigma}(\tilde{\omega}(X h_1))]\\ 
    &= \mathbb{E}[\tilde{\omega}(\Exp_{\sigma}(X h_1))]\\
    &= \tilde{\omega}\left(\mathbb{E}[\Exp_{\sigma}(X h_1)] \right)
    \end{align*}
    where for the second equality we apply \cref{lemma.involution-map-of-pre-lambda-rings} to commute $\tilde{\omega}$ with $\Exp_{\sigma}$ and the third equality follows directly from the definitions: on an element of $\Lambda_R^\wedge$, $\mathbb{E}$ is applied coefficient-wise to obtain an element of $\Lambda_C^\wedge$, i.e. it is the continuous extension of $\mathbb{E} \otimes_{\mathbb{Z}} \Id_{\Lambda}$ from $R \otimes_{\mathbb{Z}} \Lambda$ to $\Lambda_R^\wedge$, whereas the involution $\tilde{\omega}$ on $\Lambda_R^\wedge$ (resp. $\Lambda_C^\wedge$)  is the continuous extension of $\Id_{R} \otimes_{\mathbb{Z}} \tilde{\omega}|_{\Lambda}$ (resp. $\Id_{C} \otimes_{\mathbb{Z}} \tilde{\omega}|_{\Lambda}$), and evidently we have $(\mathbb{E} \otimes_{\mathbb{Z}} \Id_{\Lambda}) \circ (\Id_{R} \otimes_{\mathbb{Z}} \tilde{\omega}|_{\Lambda}) = (\Id_{C} \otimes_{\mathbb{Z}} \tilde{\omega}|_{\Lambda}) \circ (\mathbb{E} \otimes_{\mathbb{Z}} \Id_{\Lambda}).$   
\end{proof}

\section{An application to random matrices}\label{s.invariants-random-matrices}

Suppose $G$ is a compact topological group. For $R=K_0(\Rep_{\mathbb{C}} G)$, we view $R$ as a pre-$\lambda$ probability space as in \cite[\S4]{Howe.RandomMatrixStatisticsAndZeroesOfLFunctionsViaProbabilityInLambdaRings} with expectation $\mathbb{E}: R \rightarrow \mathbb{Z}$ given by integration of the trace over unit volume Haar measure; by the usual orthogonality relations, for $V \in \Rep_{\mathbb{C}} G$ with associated effective class $[V] \in K_0(\Rep_{\mathbb{C}} G)$,   
\[ \mathbb{E}[\,[V]\,] = \dim_{\mathbb{{C}}} V^{G}. \]

The ring $K_0(\Rep_{\mathbb{C}} G)$ is a pre-$\lambda$ ring with $\sigma$-operations of effective classes given by symmetric powers and $\lambda$-operations of effective classes given by exterior powers. In particular, for $V$ a complex representation of $G$, 

\[ \Exp_{\sigma}([V]h_1)=\sum_\tau  [\Sym^\tau V] m_\tau \textrm{ and } \Exp_{\sigma}(-[V]h_1)=\sum_\tau (-1)^{|\tau|} [\wedge^\tau V] m_\tau \]
where, for $\tau=\tau_1\tau_2\ldots$, $\Sym^\tau V=\otimes_i \Sym^{\tau_i} V$ \textrm{ and } $\wedge^\tau V=\otimes_i \wedge^{\tau_i} V$. For example $\tau=3\cdot2^2\cdot 1$, $\wedge^\tau V=\wedge^3 V \otimes (\wedge^2 V)^{\otimes 2} \otimes V$.

In \cite{Howe.RandomMatrixStatisticsAndZeroesOfLFunctionsViaProbabilityInLambdaRings}, we computed $\mathbb{E}[\Exp_{\sigma}([V]h_1)]$ for $V$ the standard representation of an orthogonal or symplectic group, up to a natural truncation, using the classical invariant theory of symmetric algebras as in \cite{Weyl.TheClassicalGroupsTheirInvariantsAndRepresentations}. We now apply \cref{theorem.expectation-of-negative} to obtain from this computation also the moment generating function for $-[V]$, and thus to compute the invariants of tensor powers of exterior powers of $V$.

In the following, we consider the multiplicity filtration on $\Lambda^\wedge$
\[ \Lambda^\wedge_{\mr{mult}>i} = \left\{ \sum_{||\tau||>i} a_\tau m_\tau\, |\, a_\tau \in \mathbb{Z} \right\} \]
where we recall that, for a partition $\tau$, $||\tau||$ is the number of distinct entries in $\tau$. We will also consider the degree filtration on $\Lambda^{\wedge}$, which can be written
\[ \Lambda^{\wedge}_{\mr{deg}>i}=\left\{ \sum_{|\tau|>i} a_\tau m_\tau\, |\, a_\tau \in \mathbb{Z} \right\} \]
where we recall $|\tau|$ is the sum of the partition. 
 
\begin{corollary}\label{random-matrices}\hfill
\begin{enumerate}
\item For $n \geq 1$, let $V_n$ be the standard representation of $O(n)$ on $\mathbb{C}^n$. Then, 
\[ \mbb{E}[\Exp_{\sigma}(-[V_n]h_1)]\equiv \Exp_\sigma(e_2)  \mod \tilde{\omega}(\Lambda^\wedge_{\mr{mult}>n}) \textrm{ (thus also mod $\Lambda^\wedge_{\mr{deg}>n}$)} \]
\item For $n \geq 1$ even, let $V_n$ be the standard representation of $\Sp(n)$ on $\mathbb{C}^n$. Then,
\[ \mbb{E}[\Exp_{\sigma}(-[V_n]h_1)]=\Exp_{\sigma}(h_2) \mod \tilde{\omega}(\Lambda^\wedge_{\mult>n}) \textrm{ (thus also mod $\Lambda^\wedge_{\mr{deg}>n}$)}. \]
\end{enumerate}
\end{corollary}
\begin{proof}
The arguments are identical, so we treat just (1). By \cite[Theorem 4.2.1]{Howe.RandomMatrixStatisticsAndZeroesOfLFunctionsViaProbabilityInLambdaRings}, 
\[ \mathbb{E}[\Exp_{\sigma}([V_n]h_1)] \equiv \Exp_\sigma(h_2)  \mod \Lambda^\wedge_{\mr{mult}>n}. \]
Since $\tilde{\omega}(h_2)=e_2$, Applying \cref{theorem.expectation-of-negative} and \cref{lemma.involution-map-of-pre-lambda-rings}, we find
\[ \mathbb{E}[\Exp_{\sigma}([V_n]h_1)] \equiv \Exp_\sigma(e_2)  \mod \tilde{\omega}(\Lambda^\wedge_{\mr{mult}>n}). \]
Since $\Lambda^\wedge_{\mr{mult}>n} \subseteq \Lambda^\wedge_{\mr{deg}>n}$ and the degree filtration is preserved by $\tilde{\omega}$, we conclude the congruence also holds modulo $\Lambda^\wedge_{\mr{deg}>n}$. 
\end{proof}

\begin{example}\label{remark.explicit-formulas-invariants}
As in,
\cite[Theorem A/4.2.1]{Howe.RandomMatrixStatisticsAndZeroesOfLFunctionsViaProbabilityInLambdaRings} we note 
\[ \Exp_{\sigma}(h_2)=\prod_{i\leq j}\frac{1}{1-t_it_j} \textrm{ and } \Exp_{\sigma}(e_2)=\prod_{i < j}\frac{1}{1-t_it_j}.  \] 
Thus \cref{random-matrices} gives, for any $\tau$ with $|\tau|\leq n$, explicit stable combinatorial formulas for 
\[ \dim_{\mathbb{C}} (\wedge^\tau \mathbb{C}^n)^{O(n)} \textrm{ and, for $n$ even,} \dim_{\mathbb{C}}  (\wedge^\tau \mathbb{C}^n)^{\Sp(n)}. \]
By comparing with \cite[Theorem A/4.2.1]{Howe.RandomMatrixStatisticsAndZeroesOfLFunctionsViaProbabilityInLambdaRings}, we also obtain, for $n$ even and $|\tau| \leq n$,
\begin{align*} \dim_{\mathbb{C}} (\wedge^\tau \mathbb{C}^n)^{O(n)} &= \dim_{\mathbb{C}} (\Sym^\tau \mathbb{C}^n)^{\Sp(n)} \textrm{, and} \\\dim_{\mathbb{C}} (\wedge^\tau \mathbb{C}^n)^{\Sp(n)} &= \dim_{\mathbb{C}} (\Sym^\tau \mathbb{C}^n)^{O(n)}. \end{align*}
In \cite{Howe.PerspectivesOnInvariantTheory}, invariants in tensor products of exterior algebras are studied from the perspective of highest weight theory (though these particular formulas do not, to our knowledge, appear in loc. cit. or elsewhere). By comparison, the computation given here is notable in that it uses only the most classical invariant theory for polynomial invariants (i.e. tensor products of symmetric powers) through \cite[Theorem A/4.2.1]{Howe.RandomMatrixStatisticsAndZeroesOfLFunctionsViaProbabilityInLambdaRings} followed by manipulations with symmetric functions to obtain the result (note, in particular, that our argument makes no use of the symmetric functions that are specific to the representation theory of symplectic or orthogonal groups).   
\end{example}

\begin{example}\label{example.moments}
For $M$ an $n\times n$ matrix, we consider the characteristic polynomial $P_M(t)=\det (\Id - t M)$.  Using \cref{random-matrices}, we find that, for $G=\Sp(n)$, 
\[ \int_{G} \left( \left. \left(\frac{d}{ds}\right)^r\right|_{s=0} P_M(\exp(\lambda s)) \right) d\mathrm{Haar}= \sum_{0 \leq j \textrm{ even} \leq n} (j \lambda)^r. \]
Indeed, 
\[ \det (\Id - t M) = \sum_{0 \leq j \leq n} (-1)^j e_j(\alpha_1, \ldots, \alpha_n)t^j \]
for $\alpha_i$ the eigenvalues of $M$. Plugging in $t=\exp(\lambda(s))$, we find
\[ \left( \left. \left(\frac{d}{ds}\right)^r\right|_{s=0} P_M(\exp(\lambda s)) \right) = \sum_{0 \leq j \leq n} (-1)^j (\lambda j)^r e_j(\alpha_1, \ldots, \alpha_n)  \]
Integrating $e_j$ over the Haar measure computes the dimension of the invariants in $\wedge^j \mathbb{C}^n$, which are $1$-dimensional when $j$ is even and $0$-dimensional when $j$ is odd (by applying \cref{random-matrices} and the formula for the generating function in \cref{remark.explicit-formulas-invariants}). These characteristic polynomial integrals arise in random matrix analogs of the study of the statistical behavior of moments of $L$-functions and their derivatives (and are computed more generally by other means in the literature).
\end{example}

\bibliographystyle{plain}
\bibliography{references, preprints}

\end{document}